\documentclass[12pt]{amsart}
\usepackage{latexsym,enumerate}
\usepackage{amssymb, xcolor,mathrsfs}
\usepackage{amsmath,amsthm,amsfonts,amssymb,latexsym, mathabx}
\usepackage[pagebackref,colorlinks]{hyperref}
\usepackage{arydshln}

\newtheorem{theorem}{Theorem}[section]
\newtheorem{lemma}[theorem]{Lemma}
\newtheorem{proposition}[theorem]{Proposition}

\newtheorem{theorema}{Theorem}

\theoremstyle{definition}

\newcounter{starquote}

\newenvironment{starquote}{%
  \refstepcounter{starquote}%
  \begin{quote}
  $\mathbf{(\star)}$\quad
  \begin{minipage}{0.9\linewidth}
}{%
  \end{minipage}
  \end{quote}
}

\numberwithin{equation}{section}

\newcommand{\GL}{{\mathrm {GL}}}
\newcommand{\PGL}{{\mathrm {PGL}}}
\newcommand{\SL}{{\mathrm {SL}}}
\newcommand{\PSL}{{\mathrm {PSL}}}

\newcommand{\GU}{{\mathrm {GU}}}

\newcommand{\SU}{{\mathrm {SU}}}
\newcommand{\PSU}{{\mathrm {PSU}}}

\newcommand{\Irr}{{\mathrm {Irr}}}

\newcommand{\ord}{{\mathrm {ord}}}

\newcommand{\diag}{{\mathrm {diag}}}

\newcommand{\St}{{\mathsf {St}}}

\newcommand{\lcm}{{\mathrm {lcm}}}
\newcommand{\Gal}{{\rm Gal}}

\newcommand{\Res}{{\mathrm {Res}}}

\newcommand{\QQ}{{\mathbb Q}}
\newcommand{\ZZ}{{\mathbb Z}}

\newcommand{\FF}{{\mathbb F}}

\newcommand{\ta}{\hspace{0.5mm}^{2}\hspace*{-0.2mm}}

\makeatletter \@namedef{subjclassname@2020}{\textup{2020}
Mathematics Subject Classification} \makeatother

\begin{document}

\title[Character Values and Conductors of Low-Rank Groups of Lie Type]
{Character Values and Conductors of Low-Rank Groups of Lie Type}

\author[Christopher Herbig]{Christopher Herbig}
\address{Department of Mathematical Sciences,
Northern Illinois University, Dekalb, IL 60115, USA}
\email{cherbig@niu.edu}

\author[Nguyen N. Hung]{Nguyen N. Hung}
\address{Department of Mathematics, The University of Akron, Akron,
OH 44325, USA}
\email{hungnguyen@uakron.edu}



\thanks{We thank Gabriel Navarro for several helpful comments, in particular for drawing our attention to
an example of an irreducible character
whose conductor is not one-value generated. The second author
gratefully acknowledges support from the AMS-Simons Research
Enhancement Grant (AWD-000167 AMS)}

\begin{abstract}
Let $\chi$ be a complex irreducible character of a finite group $G$.
The conductor of $\chi$, denoted $c(\chi)$, is the smallest positive
integer $n$ such that $\chi(x)\in \QQ(\exp({2\pi i/n}))$ for all
$x\in G$. We show that for certain rank $1$ finite groups of Lie
type, the conductor $c(\chi)$ is realized at a single group element;
that is, there exists $g\in G$ such that $c(\chi)=c(\chi(g))$. In
some quasisimple cases, we further prove that the field of values
\(\QQ(\chi)\) is generated by a single value. This phenomenon, which
is related to a well-known conjecture of W.~Feit, was recently
observed by Boltje \emph{et al.} in their reduction of the
conjecture to finite simple groups. Our approach uses techniques
from algebraic number theory together with the known character
tables of these groups.
\end{abstract}

\maketitle

\tableofcontents

%%%%%%%%%%%%%%%%%%%%%%%%
%%%%%%%%%%%%%%%%%%%%%%%

\section{Introduction}
This paper explores the conductor of character values for certain
finite groups of Lie type of (semisimple) rank $1$, namely the
$2$-dimensional linear groups $\GL_2(q)$ and $\SL_2(q)$, as well as
the Suzuki groups $\ta B_2(q)$.

Let \(G\) be a finite group and let \(\chi\in\Irr(G)\) be a complex
irreducible character of \(G\). For each \(x\in G\), the value
\(\chi(x)\) is a sum of roots of unity, i.e. a \emph{cyclotomic
integer}. Given a set \(S\) of cyclotomic integers, its conductor,
denoted \(c(S)\), is the smallest positive integer \(n\) such that
\(S\) is contained in the $n$-th cyclotomic field
\(\mathbb{Q}(\zeta_n)\). The \emph{conductor} of \(\chi\) is defined
by
\[
c(\chi):=c(\{\chi(x)\mid x\in G\})
=\operatorname{lcm}\bigl(\{c(\chi(x))\mid x\in G\}\bigr).
\]

%The second author~\cite{H} recently made the following observation,
%which remains unverified:
%\begin{starquote}\label{star}
%\emph{If $\chi\in\Irr(G)$, then there exists an element \(g \in G\)
%such that \(c(\chi) = c(\chi(g))\).}
%\end{starquote}
%An equivalent formulation is that there exists \(g \in G\) such that
%\(c(\chi(g))\) is divisible by \(c(\chi(x))\) for all \(x \in G\).
%This, if true, would strengthen a (still unsolved)

A well-known conjecture of W.~Feit \cite{Feit80} from 1980 asserts
that \emph{if $\chi\in \Irr(G)$, then $G$ contains an element whose
order is \(c(\chi)\)}; see \cite[\S3.3]{N} or \cite[Chapter~2]{Is18}
for modern textbook references. Some earlier results can be found in
\cite{Amit-Chillag86,Ferguson-Turull86,HTZ}. For background and
recent progress, we refer to the breakthrough work of R.~Boltje,
A.\,S.~Kleshchev, G.~Navarro, and P.\,H.~Tiep \cite{BKNT25}, where
the conjecture is reduced to a (much stronger) statement that needs
to be verified for finite simple groups.

Note that \(\chi(x)\) is a sum of \(\ord(x)\)-th roots of unity, so
\(c\bigl(\chi(x)\bigr)\mid \ord(x)\), and hence
\(c\bigl(\chi(x)\bigr)\) is equal to the order of a suitable power
of \(x\). This naturally leads to the question of whether the
following strengthening of Feit's conjecture holds: \emph{does there
exist \(g\in G\) such that $c(\chi)=c\bigl(\chi(g)\bigr)$?}
Equivalently, does there exist \(g\in G\) such that
\(c\bigl(\chi(g)\bigr)\) is divisible by \(c\bigl(\chi(x)\bigr)\)
for all \(x\in G\)? Although this fails in general, as noted in
\cite[\S9]{BKNT25}, counterexamples appear to be rare.

For finite simple groups, the character values of alternating and
sporadic groups are typically close to rational, so the statement is
easy to verify. In contrast, groups of Lie type have highly
irrational character values, making the problem substantially more
difficult. At present, no counterexamples are known for simple
groups of Lie type, or even for arbitrary finite groups of Lie type.
This leads to the problem of determining whether the following holds
true:
\begin{starquote}\label{star}
\emph{Let \(G\) be a finite group of Lie type. There exists \(g\in
G\) such that \(c(\chi)=c\bigl(\chi(g)\bigr)\)?}
\end{starquote}

The main aim of this paper is to provide initial evidence supporting
\eqref{star}. As usual, we denote by \(\Irr(G)\) the set of all
irreducible characters of \(G\).

\begin{theorema}\label{theorem:main}
Let $G$ be one of the groups $\GL_2(q)$ or $\SL_2(q)$, where $q$ is
a prime power; or $\ta B_2(q)$ where $q=2^{2n+1}$ with
$n\in\ZZ_{\ge0}$. Let $\chi\in\Irr(G)$. Then there exists an element
$g\in G$ such that
\[
c(\chi)=c(\chi(g)).
\]
\end{theorema}

One may ask whether the equality \(c(\chi)=c(\chi(g))\) can be
strengthened to \(\QQ(\chi)=\QQ(\chi(g))\), where
\(\QQ(\chi):=\QQ(\{\chi(x):x\in G\})\) denotes the field of values
of \(\chi\). We show that this is indeed the case for the
quasisimple groups \(\SL_2(q)\) and \(\ta B_2(q)\). We note that the
phenomenon of generating character fields from a single value for
finite simple groups was already observed in \cite[\S9]{BKNT25}.

We emphasize that what we actually prove in
Theorem~\ref{theorem:main} is that, for the groups under
consideration, the conductor of the value set of every irreducible
character is attained at a single value. Although the character
tables of these groups are known, this result is far from immediate.
Determining the conductor of a character's value set is generally a
nontrivial task, and establishing that it can be realized by a
single value is even more delicate. Our approach involves a careful
analysis of the conductors of irrational values, and much of the
paper therefore draws on algebraic number theory.

The proof of Theorem~\ref{theorem:main} is is given in
Propositions~\ref{prop:GL2q}, \ref{prop:main-SL2q}, and
\ref{prop:main-Suzuki}.

%%%%%%%%%%%%%%%%%%%%%%
%%%%%%%%%%%%%%%%%%%%%%%

\section{Two-dimensional general linear groups
$\GL_2(q)$}\label{sec:GL2q}

In this section we prove Theorem~\ref{theorem:main} for the family
of the two-dimensional general linear groups $\GL_2(q)$, where \(q\)
is a prime power.

For a character \(\chi\) of  a finite group \(G\), we write
\[
\mathcal{VS}(\chi) := \{\chi(x) : x \in G\},
\]
the set of all values of $\chi$. If \(\chi\) is linear, then
\(\mathcal{VS}(\chi) = \{\xi^i : i \in \ZZ_{\ge 0}\}\) for some root
of unity \(\xi\), and hence \(\chi\) trivially satisfies
(\ref{star}). When \(\chi(1) > 1\), a classical result of Burnside
implies that \(\chi\) vanishes at some group element (see
\cite[Theorem~3.15]{Is}), so \(0 \in \mathcal{VS}(\chi)\).

For simplicity, we often omit the value \(0\) from
\(\mathcal{VS}(\chi)\). More generally, for the purpose of verifying
(\ref{star}), one may disregard rational values, as well as any
value that is a rational multiple of another. Accordingly, we write
\[
\overline{\mathcal{VS}}(\chi)\subseteq \mathcal{VS}(\chi)
\]
for a \emph{reduced value set}, meaning any subset obtained by
removing such values. The statement (\ref{star}) can then be
restated as:
\begin{quote}
\emph{There exists a value $z\in \overline{\mathcal{VS}}(\chi)$ such
that \(c\left(\overline{\mathcal{VS}}(\chi)\right)=c(z)\).}
\end{quote}

%%%%%%%%%%%%%%%%%%%%%%%%%%

\subsection{Character values of $\GL_2(q)$}\label{subsec:GL2q}

The irreducible characters of \(G:=\GL_2(q)\), as well as
$\SL_2(q)$, were determined by H.~Jordan and I.~Schur in the early
20th century. We describe them here using modern language, in the
spirit of Deligne--Lusztig theory, and partially follow the notation
in \cite[\S2]{GM20}.

The group $G=\GL_2(q)$ has two conjugacy classes of maximal tori.
One consists of the maximally split tori isomorphic to $\FF_q^\times
\times \FF_q^\times$, with a representative, say $T_1$, consisting
of diagonal matrices $\{\diag(x,y): x,y\in \FF_q^\times\}$. The
other consists of non-split maximal tori of order \(q^2 - 1\); a
representative \(T_2\) is obtained, for instance, by embedding
\(\FF_{q^2}^\times\) via its linear action (by multiplication) on
the \(\FF_q\)-vector space \(\FF_{q^2}\).

The irreducible characters of \(G\) fall into four families. The
first two consists of \emph{unipotent characters}. There are \(q-1\)
linear characters, each of the form \(\theta \circ \det\) with
\(\theta \in \Irr(\FF_q^\times)\). There are also \(q-1\) characters
of degree \(q\), namely the \emph{Steinberg characters}, of the form
\(\St \otimes \theta\), where \(\St + \mathbf{1}_G\) is the
rational-valued permutation character afforded by the action of
\(G\) on the \(q+1\) one-dimensional subspaces of the natural
module. Therefore the value set of \(\St \otimes \theta\) is the
same as \(\theta\), which is of the form \(\{\xi^i : i \in \ZZ_{\ge
0}\}\), as mentioned above, for some root of unity $\xi$. In
particular, the desired conclusion holds in this case. It remains to
consider the non-unipotent characters, which turn out to be all
\emph{semisimple characters}.

The next family consists of \((q-1)(q-2)/2\) characters of degree
\(q+1\), labeled by \[X_{(m,n)}, \text{ where } 0\le m<n\le q-2.\]
These are the Deligne--Lusztig characters
\(R_{T_1}^G(\theta_1\otimes\theta_2)\), one for each unordered pair
\(\{\theta_1,\theta_2\}\) of distinct irreducible characters of
\(\FF_q^\times\). Let \(\varepsilon\) be a primitive \((q-1)\)-th
root of unity. A reduced value set of \(X_{(m,n)}\) is
\begin{equation}\label{E:X-mn}
\overline{\mathcal{VS}}(X_{(m,n)})= \left\{ \varepsilon^{(m+n)a},
\varepsilon^{nc + md} + \varepsilon^{nd + mc}:
\begin{array}{l}
0 \le a \le q-2,\\
0 \le c < d \le q-2
\end{array}
\right\}.
\end{equation}

The final family consists of \((q^2-q)/2\) characters of degree
\(q-1\). Let $E_q \subseteq \{0,1,\dots,q^2-2\}$ be a set of
representatives for the equivalence relation on $\mathbb{Z}$ defined
by $a \sim a'$ if $a' \equiv a \pmod{q^2-1}$ or $a' \equiv q a
\pmod{q^2-1}$. These characters are labeled by \[Y_{(n)}, \text{
where } n\in E_q \text{ and } (q+1)\nmid n.\] They are the
Deligne--Lusztig characters \(R_{T_2}^G(\gamma)\) with
\(\gamma\in\Irr(T_2)\), noting that
\(R_{T_2}^G(\gamma)=R_{T_2}^G(\gamma^q)\). Let \(\eta\) be a
primitive \((q^2-1)\)-th root of unity. A reduced value set of
\(Y_{(n)}\) is
\begin{equation}\label{E:Y-n}
\overline{\mathcal{VS}}(Y_{(n)})= \left\{
\eta^{na(q+1)},-(\eta^{ne}+\eta^{neq}):
\begin{array}{l}
0 \le a \le q-2,\\
e\in E_q, (q+1)\nmid e
\end{array}
\right\}.
\end{equation}

%%%%%%%%%%%%%%%%%%%%%%%%%%%%%%%%%%%%%%

\subsection{The family $X_{(m,n)}$}

We require some number-theoretic lemmas. The first of these is a
result of J.\,H.~Loxton that allows us to determine conductors of
cyclotomic integers in many scenarios.

\begin{lemma}\label{L:Lox}
Let $\alpha$ be a cyclotomic integer, and write $\alpha =
\sum_{i=1}^k \xi_i$ for roots of unity $\xi_i$. Further, assume that
$\alpha$ cannot be expressed as a sum of less than $k$  roots of
unity. Then, $\xi_i \in \QQ_{c(\alpha)}$ for each $i$. In
particular, $\QQ_{c(\alpha)} = \QQ(\xi_1,\ldots,\xi_k)$.
\end{lemma}

\begin{proof}
This follows from \cite[Theorem~1]{L75}.
\end{proof}

%\begin{proof}
%It is clear that $\QQ_{c(\alpha)} \subseteq \QQ(\epsilon_1,
%\epsilon_2)$. A result of Loxton \cite[Theorem 1]{L75} states that
%if $\beta = \sum_{i=1}^k \eta_i$ is a sum of $k$ roots of unity
%$\eta_i$, where $\beta$ cannot be expressed as a sum of fewer than
%$k$ roots of unity, then each $\eta_i$ lies in $\QQ_{c(\beta)}$.
%Thus, our assumptions now imply that $\QQ(\epsilon_1,\epsilon_2)
%\subseteq \QQ_{c(\alpha)}$. The statement now follows.
%\end{proof}

Our analysis of the characters $X_{(m,n)}$ will require us to
consider cyclotomic integers taking the form $\varepsilon^{nc + md}
+ \epsilon^{nd + mc}$ where $\varepsilon$ is a $(q-1)$-th root of
unity and $c$ and $d$ are chosen so that $nc + md = (n,m)$. Special
care will need to be taken when $\epsilon^{nc + md} + \epsilon^{nd +
mc}$ is zero, which is equivalent to having the following
congruence:
$$
nc + md \equiv nd + mc + \frac{q-1}{2} \pmod{q-1}.
$$

The next lemma shall help us handle this situation.

\begin{lemma}\label{L:2:2}
Let $n,$ $m$, and $k$ be fixed integers where $k$ is even. Further
assume that whenever $c$ and $d$ are integers such that $nc + md =
(n,m)$, then the following congruence holds:
    \begin{equation}\label{L:fund-cong}
        nc + md \equiv nd + mc + \frac{k}{2} \pmod{k}.
    \end{equation}
Then, $(n^2 - m^2)/(m,n)$ is divisible by $k$.
\end{lemma}

\begin{proof}
Given any solution $(x,y)$ to the linear Diophantine equation $nx +
my = (n,m)$, any other solution $(x',y')$ is of the form $x' = x +
tm/(n,m)$ and $y' = y - tn/(n,m)$ for some integer $t$. Thus, taking
$t = 1$, our assumptions imply that the following congruence holds:
    $$
n\bigg(c + \frac{m}{(n,m)}\bigg) + m\bigg(d - \frac{n}{(n,m)}\bigg)
\equiv n\bigg(d - \frac{n}{(n,m)}\bigg) + m\bigg(c +
\frac{m}{(n,m)}\bigg) + \frac{k}{2} \pmod{k}.
    $$
Simplifying and canceling with the congruence in
\eqref{L:fund-cong}, we are left with
    $$
        0 \equiv \dfrac{n^2 - m^2}{(m,n)} \pmod{k},
    $$
    as desired.
\end{proof}

The next lemma allows us to handle cases where $\varepsilon^{nc
+ md} + \varepsilon^{nd + mc}$ equals a single root of unity. We
take advantage of the following fact (which may be reasoned
geometrically): For any vanishing sum of three roots of unity, say
$\xi_0+ \xi_1+ \xi_2 = 0$, there is a fixed root of unity $\xi$
such that $\xi_j = \xi\zeta_3^{\pm j}$ for each $j$, where
$\zeta_3$ is a primitive third root of unity. In our case, it is
possible to factor out a $(q-1)$-th root of unity $\xi$ such that
$\varepsilon^{nc + md} = \xi\zeta_3^j$ and $\varepsilon^{nd + mc} =
\xi\zeta_3^{2j}$ for $j \in \{1,2\}$. This yields that
$$
    \epsilon^{nc + md} + \epsilon^{nd + mc} = \xi(\zeta_3^j + \zeta_3^{2j}) = -\xi.
$$
This implies the following congruence holds:
$$
    nc + md  \pm \dfrac{q-1}{3} \equiv nd + mc \mp  \dfrac{q-1}{3} \pmod{q-1},
$$
which may be rewritten as
$$
    nc + md  \equiv nd + mc + \sigma_{c,d}\cdot \dfrac{q-1}{3} \pmod{q-1},
$$
where $\sigma_{c,d}$ is a sign depending upon a choice of integers
$c$ and $d$ such that $nc + md = (n,m)$. Ideally, we would like to
have a result analogous to Lemma~\ref{L:2:2}, but the fact that
$\sigma_{c,d}$ depends upon a choice of $c$ and $d$ poses an
obstruction. However, we show in our next lemma that under our
circumstances, the sign is actually independent of our choice of $c$
and $d$.

\begin{lemma}\label{L:2:4}
Let $n$, $m$, and $k$ be fixed integers where $3 \mid k$. Assume for
any integers $c$ and $d$ such that $nc + md = (n,m)$, there exists a
sign $\sigma_{c,d} \in \{-1,1\}$ such that
    $$
    nc + md \equiv nd + mc + \sigma_{c,d}\cdot\frac{k}{3} \pmod{k}.
    $$
Then, in fact, there exists a fixed sign $\sigma \in \{-1,1\}$ such
that
    $$
    nc + md \equiv nd + mc + \sigma\cdot \frac{k}{3} \pmod{k}
    $$
for all integers $c$ and $d$ with $nc + md = (n,m)$. In particular,
$(n^2 - m^2)/(n,m)$ is divisible by $k$.
\end{lemma}

\begin{proof}
Fix an integer solution $(c,d)$ to the linear Diophantine equation
$nc + md = (n,m)$. Define integer functions $c(t)$ and $d(t)$ by
setting $c(t) = c + tm/(n,m)$ and $d(t) = d - tn/(n,m)$. We have
that
    $$
        -\sigma_{c(t),d(t)}\frac{k}{3} + nc(t)+ md(t) \equiv nd(t) + mc(t) \pmod{k},
    $$
    which simplifies to
    \begin{equation}\label{E:plac1}
        -\sigma_{c(t),d(t)}\frac{k}{3} + nc + md \equiv nd + mc + \dfrac{t(m^2 - n^2)}{(n,m)} \pmod{k}.
    \end{equation}
Take $t = 1$. Notice that since $nc + md - nd - mc \equiv
\sigma_{c,d}\cdot\frac{k}{3} \pmod{k}$, it follows that
    $$
    \dfrac{(m^2 - n^2)}{(n,m)} \equiv s\cdot\frac{k}{3} \pmod{k},
    $$
for some $s \in \{0,1,2\}$. However, if $3 \nmid s$, then among all
values of $t \in \ZZ$, the right hand side of \eqref{E:plac1} takes
three distinct values whereas the left hand side takes at most two
distinct values, an impossibility. Thus, $3 \mid s$, and it follows that $(n^2 -
m^2)/(m,n)$ is divisible by $k$. Using this fact in \eqref{E:plac1},
we obtain the following:
    $$
    \sigma_{c,d} \equiv nc + md - nd - mc \equiv \sigma_{c(t),d(t)} \pmod{k}
    $$
for all $t \in \ZZ$, thus verifying the first claim.
\end{proof}

\begin{lemma}\label{L:2:3}
Let $k$ and $r$ be fixed integers, and let $n$ and $m$ be integers
such that $k$ divides $(n^2 - m^2)/(n,m)$. Further, assume that
whenever $c$ and
    $d$ are integers such that $nc + md = (n,m)$, then the following
    congruence holds:
    \begin{equation}\label{L:fund-cong22}
        nc + md \equiv nd + mc + r \pmod{k}.
    \end{equation}
If $c'$ and $d'$ are integers such that $nc' + md' = \ell\gcd(n,m)$
for some integer $\ell$, then we have
    $$
    nc' + md' \equiv nd' + mc' + \ell r \pmod{k}.
    $$
\end{lemma}

\begin{proof}
    We have that $nc' + md' = \ell(n,m) = nc\ell + nd\ell$. Thus, there
    exists an integer $t$ such that $c' = c\ell + tm/(n,m)$ and $d' =
    d\ell - tn/(n,m)$. It now follows that
    $$
    nd' + mc' \equiv nd\ell + mc\ell - t\cdot \dfrac{n^2 - m^2}{(n,m)} \pmod{k}.
    $$
    Notice that on the rightmost side above, the term $(n^2 -
    m^2)/(n,m)$ vanishes by assumption. Now, by
    \eqref{L:fund-cong22}, the rightmost side is therefore congruent to
    $nc\ell + md\ell - \ell r$. But since $nc\ell + md\ell =
    nc' + md'$, it now follows that
    $$
    nc' + nd' - \ell r \equiv nd' + mc' \pmod{k},
    $$
    and the lemma is proved.
\end{proof}

\begin{proposition}\label{P:Xmn}
Theorem~\ref{theorem:main} holds for the characters $X_{(m,n)}$
($0\le m<n\le q-2$) of $\GL_2(q)$.
\end{proposition}

\begin{proof}
Fix integers $m$ and $n$, and let $\varepsilon$ be a primitive
$(q-1)$-th root of unity, as before. Given an integer $t$, the order
of $\varepsilon^t$ is given by $(q-1)/(q-1,t)$. In particular,
notice that among all roots of unity $\varepsilon^{nc + md}$ for
integers $c$ and $d$, the order is maximized when $c$ and $d$ are
chosen such that $nc + md = \ell(n,m)$ where $\ell$ is some integer
relatively prime to $q - 1$. Notice that for any integers $c'$ and
$d'$, the root $\varepsilon^{nc' + md'}$ is equal to some power of
$\varepsilon^{nc + md}$. We have that $\QQ(\varepsilon^{nc' + md'})
\subseteq \QQ(\varepsilon^{nc + md})$ for all integer pairs
$(c',d')$, and most importantly, by considering the values taken by
$X_{(m,n)}$ in \eqref{E:X-mn}, we conclude \[\QQ_{c(X_{(m,n)})}
\subseteq \QQ(\varepsilon^{nc + md}).\]

\medskip

(I) If there exists $c$ and $d$ such that $nc + md = (n,m)$ and
$\varepsilon^{nc + md} + \varepsilon^{nd + mc}$ is neither zero nor
expressible as a single root of unity, then, by Lemma~\ref{L:Lox},
    $$
\QQ(\varepsilon^{nc + md}) = \QQ_{c(\varepsilon^{nc + md} +
\varepsilon^{nd + mc})} \subseteq \QQ_{c(X_{(m,n)})}.
    $$
We deduce that $c(X_{(m,n)})  = c(\varepsilon^{nc + md} +
\varepsilon^{nd + mc})$, and the proposition follows in this case.

For the remainder of the proof, we shall assume that for each
integer $c$ and $d$ with $nc + md = (n,m)$, the sum $\varepsilon^{nc
+ md} + \varepsilon^{nd + mc}$ is either zero or a single root of
unity.

\medskip

(II) Consider first the case where for any integers $c$ and $d$ such
that $nc + md = (n,m)$, the sum $\varepsilon^{nc + md} +
\varepsilon^{nd + mc} = 0$. Equivalently, we have the following
congruence:
    $$
    nc + md \equiv nd + mc +\dfrac{q-1}{2}\pmod{q-1}.
    $$
This is precisely the assumption in Lemma~\ref{L:2:2} with $k =
q-1$. (The assumption that $2\mid(q - 1)$ is satisfied since
$\varepsilon^{nc + md} = -\varepsilon^{nd + mc}$, which forces one
of these two roots to have even order.) %We claim that
%the conductor is attained at the value $\varepsilon^{2nc +
%2md}+\varepsilon^{2nd + 2mc}$.
In particular, Lemma~\ref{L:2:3} applies with $r = (q-1)/2$, and we
conclude that whenever $nc' + md'$ is an odd multiple of $(n,m)$, we
have $\varepsilon^{nc' + md'} + \varepsilon^{nd' + mc'} = 0$. On the
other hand, if $nc' + md'$ is an even multiple of $(n,m)$, then the
same lemma implies that $\varepsilon^{nc' + md'} + \varepsilon^{nd'
+ mc'} = 2\varepsilon^{nc' + md'}$. Notice also that whenever $nc' +
md'$ is an even multiple of $(n,m)$, then we have that
    $$
    \dfrac{q-1}{(q-1,2(n,m))} \cdot \dfrac{(q-1,nc' + md')}{q-1} \in \ZZ.
    $$
We conclude that each such root of the form $\varepsilon^{nc' +
md'}$ is some power of $\varepsilon^{2nc + 2md}$. In particular, the
values $\varepsilon^{nx + my}+\varepsilon^{ny + mx}$ of $X_{(m,n)}$
all lie within $\QQ(\varepsilon^{2nc + 2md})$. It remains to show
that the values $\varepsilon^{(m+n)x}$ of $X_{(m,n)}$ also lie in
$\QQ(\varepsilon^{2nc + 2md})$. To this end, it suffices to show
that $n + m$ is an even multiple of $(n,m)$. Were this not true,
then by invoking Lemma~\ref{L:2:3} with $c' = d' = 1$, this would
imply that $n + m \equiv n + m + (q-1)/2 \pmod{q-1}$, an
impossibility. Now, the value $\varepsilon^{2nc + 2md}$ is
attained by $X_{(m,n)}$ if $2c \equiv 2d \pmod{q-1}$, or the value
$2\varepsilon^{2nc + 2md}$ is attained at $\varepsilon^{2nc +
2md}+\varepsilon^{2nd + 2mc}$ otherwise. We conclude that
$\QQ(X_{(m,n)}) = \QQ(\varepsilon^{2nc + 2md})$, verifying the
proposition in this case.

\medskip

(III) Now, we assume that for any $c$ and $d$ with $nc + md =
(n,m)$, we have $\varepsilon^{nc + md} + \varepsilon^{nc + md} =
-\xi_{c,d}$ for some root of unity $\xi_{c,d}$. In particular, this
gives us a vanishing sum of three roots of unity:
    $$
    \varepsilon^{nc + md} + \varepsilon^{nd + mc} + \xi_{c,d} = 0.
    $$
We claim that $\QQ(X_{(m,n)}) = \QQ(\xi_{c,d})$. Notice that if
$\QQ(\xi_{c,d}) = \QQ(\varepsilon^{nc + md})$, then the result
follows in this case by our remarks in the first paragraph.
Otherwise, we have that $[\QQ(\varepsilon^{nc + md}):
\QQ(\xi_{c,d})] = 3$. For any vanishing sum of three roots of
unity, we may factor out a root of unity $\xi$ which yields
    $$
    \xi + \xi\zeta_3+ \xi\zeta_3^2 = 0,
    $$
where $\zeta_3$ is a primitive third root of unity. In particular,
we have that
    $$
    \xi_{c,d}:= \varepsilon^{nc + md \pm (q-1)/3} = \varepsilon^{nd + mc \mp (q-1)/3}.
    $$
for any choice of $c$ and $d$ where $nc + md = (n,m)$. Equivalently,
we have that
    $$
    nc + md \equiv nd + mc + \sigma_{c,d}\cdot \dfrac{q-1}{3} \pmod{q-1}
    $$
for some sign $\sigma_{c,d}$ depending on $c$ and $d$. However, by
Lemma~\ref{L:2:4}, the sign $\sigma_{c,d}$ is actually independent
from $c$ and $d$, and we may write $\sigma := \sigma_{c,d}$:
    $$
    nc + md \equiv nd + mc + \sigma\cdot \frac{q-1}{3} \pmod{q-1}
    $$
Thus, the assumptions of Lemma~\ref{L:2:3} are satisfied with $k = q-1$ and $r = \sigma(q-1)/3$. In
particular, let $c'$ and $d'$ be integers where $nc' + md' =
\ell(n,m)$ for some integer $\ell$. One verifies using
Lemma~\ref{L:2:4} that $\varepsilon^{nd' + mc'} =
\varepsilon^{nd\ell + mc\ell}$, and so
    $$
    \varepsilon^{nc' + md'} + \varepsilon^{nd' + mc'} = \varepsilon^{nc\ell + md\ell} +
    \varepsilon^{nd\ell + mc\ell} = \xi_{c,d}^\ell(\zeta_3^\ell + \zeta_3^{2\ell}).
    $$
The rightmost side above equals $2\xi_{c,d}^\ell$ if $\ell \mid 3$
and $-\xi_{c,d}^\ell$ otherwise. In either case, $\varepsilon^{nc' +
md'} + \varepsilon^{nd' + mc'} \in \QQ(\xi_{c,d})$. This shows that
all values $\varepsilon^{nx + my} + \varepsilon^{ny + mx}$ of
$X_{(m,n)}$ lie in $\QQ(\xi_{c,d})$. It remains to show that other
values of $X_{(m,n)}$ also lie in $\QQ(\xi_{c,d})$. To this end, it
suffices to show that $n + m = \ell(n,m)$ where $3 \mid \ell$.
However, by Lemma~\ref{L:2:3}, we have $n + m \equiv n + m +
\frac{\sigma\ell (q-1)}{3} \pmod{q-1}$, which can only hold if $3
\mid \ell$. This proves the result in this case.

\medskip

(IV) The last remaining case is where certain choices of $c$ and $d$
with $nc + md = (n,m)$ yield $\varepsilon^{nc + md} +
\varepsilon^{nd + mc} = 0$ and other choices result in
$\varepsilon^{nc + md} +
\varepsilon^{nd + mc}$ equaling a single root of unity. %We claim
%that this assumption allows us to choose a group element $g$ with
%$\QQ(\chi_{q+1}^{(n,m)}) = \QQ(\epsilon^{na + mb})$.
Let $c_1, c_2, d_1$, and $d_2$ be integers where $nc_1 + md_1 = nc_2
+ md_2 = (n,m)$. Further assume that $\varepsilon^{nc_1 + md_1} +
\varepsilon^{nd_1 + mc_1} = 0$ and that $\varepsilon^{nc_2 + md_2} +
\varepsilon^{nd_2 + mc_2}$ equals a single root of unity. We obtain
the following congruences as results:
    \begin{equation}\label{E:cong-one4}
        nc_1 + md_1 \equiv nd_1 + mc_1 + \dfrac{q-1}{2} \pmod{q-1},
    \end{equation}
    \begin{equation}\label{E:cong-two4}
        nc_2 + md_2 \equiv nd_2 + mc_2 \pm \frac{q-1}{3} \pmod{q-1}.
    \end{equation}
We may choose integers $\ell_1$ and $\ell_2$ satisfying the
following properties: $\ell_1$ is odd, $3 \nmid \ell_2$, and
$(q-1,\ell_1 + \ell_2) = 1$. For instance, we may take $\ell_1 = 1$
and $\ell_2 = 2^{2^j}$ for $j$ sufficiently large. Numbers of the
form $2^{2^j} + 1$ are called Fermat numbers, and any two distinct
Fermat numbers are relatively prime, so we may choose $j$ large
enough where $2^{2^j} +1$ shares no factors in common with $q - 1$. We
now evaluate the following value of $X_{(n,m)}$:
    \begin{equation}\label{E:comb-cong}
\alpha := \varepsilon^{ n(c_1 + 2^{2^j}c_2) + m(d_1 + 2^{2^j}d_2)} +
\varepsilon^{n(d_1 + 2^{2^j}d_2) + m(c_1 + 2^{2^j}c_2)}.
    \end{equation}
Now, multiplying the congruence \eqref{E:cong-two4} by $2^{2^j}$ and
adding \eqref{E:cong-one4}, we have that
   \[\begin{aligned}
   n(c_1 + 2^{2^j}c_2) &+ m(d_1 + 2^{2^j}d_2)\equiv\\
     & n(d_1 + 2^{2^j}d_2) + m(c_1 + 2^{2^j}c_2) + \dfrac{q-1}{2} \pm 2^{2^j}\cdot \dfrac{q-1}{3} \pmod{q-1}.
  \end{aligned}\]
Combining the last two terms above, we obtain
    \[\begin{aligned}
    n(c_1 + 2^{2^j}c_2) &+ m(d_1 + 2^{2^j}d_2) \equiv\\
    &n(d_1 + 2^{2^j}d_2) + m(c_1 + 2^{2^j}c_2) + \sigma\cdot\dfrac{q-1}{6} \pmod{q-1}
    \end{aligned}\]
for some $\sigma \in \{-1,1\}$. In particular, if
$\alpha$ were to be either zero or a root of unity, then
the constant term on the right hand side above must either be $(q-1)/2$ or
$\pm (q-1)/3$, which is not the case. Thus, $\alpha$ is
nonzero and is not equal to a root of unity. Now, observe that
$\varepsilon^{ n(c_1 + 2^{2^j}c_2) + m(d_1 + 2^{2^j}d_2)} =
\varepsilon^{(2^{2^j} + 1)(n,m)}$ has order equal to
    $$
    \dfrac{q-1}{(q-1, (2^{2^j} + 1)(n,m))} = \dfrac{q-1}{(q-1, (n,m))}
    $$
where the above equality holds since $(q-1, 2^{2^j} + 1) = 1$ by our
choice of $j$. In particular, by Lemma~\ref{L:Lox}, we have that
$\QQ(\varepsilon^{(n,m)}) = \QQ_{c(\alpha)} \subseteq
\QQ_{c(X_{(m,n)})}$, which completes the proof.
\end{proof}

%%%%%%%%%%%%%%%%%%%%%%%%%%%%%%%%%%%%%%%%

\subsection{The family $Y_{(n)}$}

\begin{proposition}\label{P:Yn}
Theorem~\ref{theorem:main} holds for the characters $Y_{(n)}$
(\(n\in E_q\) and \((q+1)\nmid n\)) of $\GL_2(q)$.
\end{proposition}

\begin{proof}
Fix an integer $n\in E_q$ such that \((q+1)\nmid n\), and recall
that $\eta$ is a primitive $(q^2-1)$-th root of unity. By
considering the value set for $Y_{(n)}$ in \eqref{E:Y-n}, we see
that $\QQ_{c(Y_{(n)})} \subseteq \QQ(\eta^n)$. Observe that if
$\eta^n + \eta^{nq}$ is neither zero nor a single root of unity,
then by Lemma \ref{L:Lox}, we have that $\QQ(\eta^n) = \QQ_{c(\eta^n
+ \eta^{nq})} \subseteq \QQ_{c(Y_{(n)})}$, and thus, the proposition
holds under these assumptions. We shall assume for the remainder of
the proof that $\eta^n + \eta^{nq}$ is either zero or equal to a
single root of unity.

First assume that $\eta^n + \eta^{nq} = 0$. This assumption is
equivalent to the following congruence:
    $$
        n \equiv nq + \dfrac{q^2 - 1}{2} \pmod{q^2 - 1}.
    $$
In particular, this forces $q^2 - 1$ to be even since $\eta^n =
-\eta^{nq}$. Consider the values of $Y_{(n)}$ of the form
$-\eta^{ne} - \eta^{nqe}$. Multiplying the above congruence by $e$,
one sees that $-\eta^{ne} - \eta^{nqe} = 0$ if $e$ is odd and
$-\eta^{ne} - \eta^{nqe} = -2\eta^{ne}$ if $e$ is even. In
particular, all these values of $Y_{(n)}$ lie in the field
$\QQ(\eta^{2n})$, and we claim the conductor of $Y_{(n)}$ is
attained at the value $-\eta^{2n} - \eta^{2nq}$. It therefore
suffices to show that $\eta^{n(q+1)} \in \QQ(\eta^{2n})$. However,
we have established that $q^2 - 1$ is even, which forces $q+1$ to be
even as well. Thus, $\eta^{n(q+1)}$ is a power of $\eta^{2n}$. We
have shown that all values of $Y_{(n)}$ lie in $\QQ(\eta^{2n}) =
\QQ(-\eta^{2n} - \eta^{2nq})$. The result is now proven in this
case.

We now complete the proof by considering the case where $\eta^{n} +
\eta^{nq}$ equals some root of unity, say $-\xi_n$. As in the proof
of Proposition~\ref{P:Xmn}, this implies that we may write $\eta^{n}
= \xi_n\zeta_3^{j}$ and $\eta^{nq} = \xi_n\zeta_3^{2j}$ where
$\zeta_3$ is a primitive third root of unity and $j \in \{1,2\}$. We
obtain the following congruence from this relation:
    $$
        n \pm \dfrac{q^2 - 1}{3} \equiv nq \mp \dfrac{q^2 - 1}{3} \pmod{q^2 - 1},
    $$
which we may rewrite as
    \begin{equation}\label{E:Y-cong}
        n  \equiv nq \pm \dfrac{q^2 - 1}{3} \pmod{q^2 - 1}.
    \end{equation}
Multiplying both sides above by some integer $e$, we obtain
    $$
        ne  \equiv nqe \pm e\cdot\dfrac{q^2 - 1}{3} \pmod{q^2 - 1},
    $$
which implies that the value $-(\eta^{ne}+\eta^{nqe})$ of $Y_{(n)}$
is either $-\xi_{n}^e$ if $3 \nmid e$ and $2\xi_n^{e}$ if $3
\mid e$. In particular, all these values lie in the field
$\QQ(\xi_n) = \QQ(-(\eta^n+\eta^{nq}))$.

Still assuming that $\eta^{n} + \eta^{nq} = -\xi_n$, it now
suffices to show that $\eta^{n(q+1)} \in \QQ(\xi_{n})$ to complete
the proof. We claim that if suffices to show that $q+1$ is divisible
by 3. First, notice that for any integer $c$, we have that
$\eta^{3nc} = (\xi_n\zeta_3^j)^{3c} = \xi_n^{3c}$. We observe
that if $3 \mid (q+1)$, then $\eta^{n(q+1)}$ is a power of
$\xi_n$, as desired. Now, observe that our assumption forces
$\eta^n$ to have order divisible by 3, implying that $3 \mid (q^2 -
1)$. In particular, $q$ is not a power of 3. If $q + 1$ is
indivisible by 3, then we necessarily have that $q - 1$ is divisible
by 3. We take \eqref{E:Y-cong} and multiply through by $q+1$:
    $$
        n(q+1) \equiv nq(q+1) \pm (q + 1)\cdot\dfrac{q^2 - 1}{3} \pmod{q^2 - 1}.
    $$
Rearranging, we obtain:
    $$
        nq + n \equiv nq^2 + nq \pm 2\cdot\dfrac{q^2 - 1}{3}  + (q - 1)\cdot\dfrac{q^2 - 1}{3} \pmod{q^2 - 1}.
    $$
Notice that as $q-1$ is divisible by 3, the last term on the right
hand side above vanishes. Further rearranging, we obtain
    $$
        n(1 - q^2) \equiv  \mp \dfrac{q^2 - 1}{3} \pmod{q^2 - 1}.
    $$
The left hand side above is zero modulo $q^2 - 1$, whereas the right
hand side is certainly nonzero modulo $q^2 - 1$, a contradiction.
Thus, $3 \mid (q+1)$, and the proposition holds in this case as
well.
\end{proof}

\begin{proposition}\label{prop:GL2q}
Theorem~\ref{theorem:main} holds for all irreducible characters of
$\GL_2(q)$.
\end{proposition}

\begin{proof}
This follows from Propositions~\ref{P:Xmn} and \ref{P:Yn} and the
discussion at the beginning of Subsection~\ref{subsec:GL2q}.
\end{proof}

%%%%%%%%%%%%%%%%%%%%%%%%%%%%%%%%%%%%%%%%%%
%%%%%%%%%%%%%%%%%%%%%%%%%%%%%%%%%%%%%%%%%%%%

\section{Two-dimensional special linear groups
$\SL_2(q)$}\label{sec:SL2q}

\subsection{Character values of $\SL_2(q)$}
Following \cite[\S2.1]{GM20}, we describe the irreducible characters
of $G=\SL_2(q)$ through the restrictions of those of $\GL_2(q)$.

The group \(G = \SL_2(q)\) has two unipotent characters: the trivial
character and the Steinberg character. Both arise as restrictions of
the corresponding linear and Steinberg characters of \(\GL_2(q)\),
and take only rational values; thus there is nothing to prove in
these cases. It remains to consider the irreducible constituents of
the restrictions of the characters \(X_{(m,n)}\) and \(Y_{(n)}\).

The multiplication of characters induces natural actions of the
group of linear characters of $\GL_2(q)$ on the sets $\{X_{(m,n)}\}$
and $\{Y_{(n)}\}$. Characters in the same orbit have the same
restrictions to $\SL_2(q)$ and the distinct restrictions are:
\[
\begin{aligned}
X'_{(k)} &:= \Res_{\SL_2(q)}^{\GL_2(q)}\!\bigl(X_{(0,k)}\bigr) \text{ for } 1 \le k \le \lfloor (q-1)/2 \rfloor, \\
Y'_{(n)} &:= \Res_{\SL_2(q)}^{\GL_2(q)}\!\bigl(Y_{(n)}\bigr) \text{
for } 1 \le n \le \lfloor (q+1)/2 \rfloor.
\end{aligned}
\]

When \(q\) is odd and \(k = (q-1)/2\), the character
\(X'_{((q-1)/2)}\) is rational-valued and decomposes as a sum of two
irreducible characters of \(\SL_2(q)\), both having the same field
of values \(\QQ(\sqrt{(-1)^{(q-1)/2}q})\). A similar phenomenon
occurs for the character \(Y'_{((q+1)/2)}\). Since
\(\QQ(\sqrt{(-1)^{(q-1)/2}q})\) is quadratic, the result follows
immediately.

Except for the above special cases, the characters \(X'_{(k)}\) and
\(Y'_{(n)}\) are all irreducible. Their (reduced) value sets can be
derived directly from those of \(X_{(0,k)}\) and \(Y_{(n)}\). In
particular, a reduced value set for
\[
X'_{(k)}, \text{ where } 1 \le k \le \lfloor (q-2)/2 \rfloor,
\]
is
\begin{equation}\label{E:X'SL2q}
\overline{\mathcal{VS}}(X'_{(k)}) = \{\varepsilon^{ka} +
\varepsilon^{-ka} : 1 \le a \le \lfloor (q-2)/2 \rfloor\},
\end{equation}
where \(\varepsilon\) is a primitive \((q-1)\)-th root of unity, as
in the case of \(\GL_2(q)\). Also, a reduced value set for
\[
Y'_{(n)}, \text{ where } 1 \le n \le \lfloor q/2 \rfloor,
\]
is
\begin{equation}\label{E:Y'SL2q}
\overline{\mathcal{VS}}(Y'_{(n)}) = \{\zeta^{nb} + \zeta^{-nb} : 1
\le b \le \lfloor (q-1)/2 \rfloor\},
\end{equation}
where \(\zeta\) is a primitive \((q+1)\)-th root of unity.

\subsection{Proof of Theorem~\ref{theorem:main} for $\SL_2(q)$}

\begin{lemma}\label{L:diss-lem}
Let $\alpha$ be a cyclotomic integer expressible as a sum of two
roots of unity, say $\alpha = \varepsilon_1 + \varepsilon_2$, and
further, assume that $\alpha$ is neither zero nor equal to a single
root of unity. Then $[\QQ_{c(\alpha)}:\QQ(\alpha)] \leq 2$.
\end{lemma}

\begin{proof}
Choose $\sigma \in \Gal(\QQ_{c(\alpha)}/\QQ(\alpha))$. By
Lemma~\ref{L:Lox}, both $\varepsilon_1$ and $\varepsilon_2$ lie in
the field $\QQ_{c(\alpha)}$. Thus, $\varepsilon_1^\sigma$ and
$\varepsilon_2^\sigma$ are well-defined roots of unity. We have that
    $$
\varepsilon_1 + \varepsilon_2 - \varepsilon_1^\sigma -
\varepsilon_2^\sigma = 0.
    $$
This is a vanishing sum of four roots of unity. One may reason
geometrically that either $\varepsilon_1^\sigma = \varepsilon_1$ or
$\varepsilon_1^\sigma = \varepsilon_2$. Likewise, either
$\varepsilon_1^\sigma = \varepsilon_1$ or $\varepsilon_1^\sigma =
\varepsilon_2$. In particular, $\sigma$ either fixes or swaps the
terms of $\alpha$. As $\sigma$ was arbitrary and as $\sigma$ is
determined by its action on $\varepsilon_1$ and $\varepsilon_2$, the
lemma now follows.
\end{proof}

The following lemma handles both families of characters $X'_{(k)}$
and $Y'_{(n)}$.

\begin{lemma}\label{lem:SL2q}
    Let $n$ be a fixed integer, and let $\varepsilon$ be any
    root of unity. If we let $S$ denote the set
    $\{\varepsilon^{na} + \varepsilon^{-na} : a \in \ZZ\}$,
    then $\QQ(S) = \QQ(\varepsilon^{n} + \varepsilon^{-n})$.
\end{lemma}

\begin{proof}
    If $\varepsilon^{n} + \varepsilon^{-n}$ is neither zero
    nor equal to a single root of unity, then by
    Lemma~\ref{L:Lox}, we have that
$\QQ_{c(\varepsilon^{n} + \varepsilon^{-n})} = \QQ(\varepsilon^n)$.
If $\varepsilon^{n} + \varepsilon^{-n}$ is rational, then the lemma
of course holds. Otherwise, by the above lemma,
$\Gal(\QQ(\varepsilon^n) / \QQ(\varepsilon^{n} + \varepsilon^{-n}))$
has one nonidentity automorphism, this being complex conjugation. In
particular,  $\QQ(\varepsilon^{n} + \varepsilon^{-n})$ is exactly
the subfield of $\QQ(\varepsilon^n)$ fixed by complex conjugation.
As all other elements of $S$ are fixed by conjugation, they too must
lie in $\QQ(\varepsilon^{n} + \varepsilon^{-n})$, proving the lemma
in this case.

    On the other hand, if $\varepsilon^{n} + \varepsilon^{-n} = 0$,
    this implies that $\varepsilon^n = \pm i$.
    In this case, it follows that
    $$
    \{\varepsilon^{na} + \varepsilon^{-na} : a \in \ZZ\} = \{0,2,-2\},
    $$
    and the result holds trivially in this case.

    Lastly, if $\varepsilon^{n} + \varepsilon^{-n}$ equals a
    root of unity, then this root must be $\pm 1$ since
    $\varepsilon^{n} + \varepsilon^{-n}$ is real. We conclude
    that $\varepsilon^{n}$ is either a primitive third or
    primitive sixth root of unity, and it follows that
    $$
    \{\varepsilon^{na} + \varepsilon^{-na} : a \in \ZZ\} \subseteq \{\pm 1, \pm 2\},
    $$
    and the result again holds trivially.
\end{proof}

\begin{proposition}\label{prop:main-SL2q}
Theorem~\ref{theorem:main} holds for all irreducible characters of
$\SL_2(q)$.
\end{proposition}

\begin{proof}
By the discussion in this section, it suffices to prove the result
for the characters $X'_{(k)}$ and $Y'_{(n)}$, whose value sets are
given in \eqref{E:X'SL2q} and \eqref{E:Y'SL2q}. The claim then
follows from Lemma~\ref{lem:SL2q}.
\end{proof}

%%%%%%%%%%%%%%%%%%%%%%%%%%%%%%%%%%%%%%%%%%
%%%%%%%%%%%%%%%%%%%%%%%%%%%%%%%%%%%%%%%%%%%%

\section{The Suzuki groups $\ta B_2(q)$}\label{sec:suzuki}

In this section, we prove Theorem~\ref{theorem:main} for the family
of Suzuki groups \(G = {}^{2}B_2(q)\), where \(q = 2^{2n+1}\) with
\(n \in \mathbb{Z}_{\ge 0}\). The case \(q=2\) corresponds to
\({}^{2}B_2(2)\), which is a Frobenius group of order \(20\). In
this case, all irreducible characters have fields of values
\(\mathbb{Q}\) or \(\mathbb{Q}(i)\), and the result follows
immediately. Thus, we assume \(n \ge 1\).

\subsection{Character values of $\ta B_2(q)$}\label{subsec:Suzuki-characters}
The irreducible characters of \(G\) were determined by Suzuki in
\cite{Suzuki62}. There are four unipotent characters, whose fields
of values are again contained in \(\mathbb{Q}(i)\), as in the case
\(n=0\). It therefore suffices to consider the non-unipotent
characters. These are all semisimple characters, and in particular
they take integer values on unipotent elements. Therefore, for a
reduced value set of such character, we can ignore the values on
unipotent elements. On the other hand, all non-unipotent elements of
\(G\) are semisimple. Consequently, it is sufficient only to examine
the values of semisimple characters on semisimple elements.

There are three conjugacy classes of maximal tori in \(G\). Let
\(T_1\), \(T_2\), and \(T_3\) be representatives of these classes,
of orders \(q-1\), \(q-\sqrt{2q}+1\), and \(q+\sqrt{2q}+1\),
respectively. Note that these tori are all cyclic. Every semisimple
element of \(G\) is conjugate to an element of \(T_1 \cup T_2 \cup
T_3\). The normalizers of these tori in \(G\) are Frobenius groups
of orders \(2|T_1|\), \(4|T_2|\), and \(4|T_3|\), respectively. It
follows that each nonidentity semisimple element of \(G\) is
conjugate to exactly two elements of \(T_1\), or to four elements of
\(T_2\), or to four elements of \(T_3\).

The first family of semisimple characters of \(G\) consists of
\((q-2)/2\) Lusztig induced (Harish--Chandra induced, indeed)
characters \(R_{T_1}^G(\theta)\), where \(\textbf{1}_{T_1}\neq\theta
\in \Irr(T_1)\), each of degree \(q^2+1\). (Note that \(\theta\) is
\(G\)-conjugate to \(\theta^{-1}\) and \(R_{T_1}^G(\theta) =
R_{T_1}^G(\theta^{-1})\).) We label these characters by
\[
X_{(n)} \text{ for } 1 \le n \le (q-2)/2.
\]
They vanish on nontrivial elements of \(T_2\) and \(T_3\). A reduced
value set of \(X_{(n)}\) is given by
\begin{equation}\label{eq:SuX}
\overline{\mathcal{VS}}(X_{(n)})=\{\varepsilon^{na} +
\varepsilon^{-na} : 1 \le a \le q-2\},
\end{equation}
where \(\varepsilon\) is a primitive \((q-1)\)-th root of unity.

The next family of semisimple characters of $G$ consists of
$(q+\sqrt{2q})/4)$ characters $R_{T_2}^G(\theta)$ where
$\textbf{1}_{T_2}\neq\theta\in \Irr(T_2)$, of degree
$(q-1)(q-\sqrt{2q}+1)$. Here each $\theta$ is $G$-conjugate to
$\theta^{-1}$ and $\theta^{\pm q}$. We therefore label them by
\[
Y_{(m)} \text{ for } m\in E_q,
\]
where $E_q \subseteq\{0,1,\dots,(q+\sqrt{2q})\}$ is a set of
representatives for the equivalence relation on $\mathbb{Z}$ defined
by $b \sim b'$ if $b' \equiv b \pmod{q+\sqrt{2q}+1}$ or $b' \equiv
qb \pmod{q+\sqrt{2q}+1}$. These characters vanish on nontrivial
elements of $T_1$ and $T_3$ and a reduced value set of $Y_{(m)}$ is
\begin{equation}\label{eq:SuY}
\overline{\mathcal{VS}}(Y_{(m)})=\{-(\zeta^{mb}+\zeta^{mbq}+\zeta^{-mb}+\zeta^{-mbq}):
1\le b\le q+\sqrt{2q}\},
\end{equation}
where $\zeta$ is a primitive \((q+\sqrt{2q}+1)\)-th root of unity.

The last family of semisimple characters of $G$ consists of
$(q-\sqrt{2q})/4)$ characters $R_{T_3}^G(\theta)$ where
$\textbf{1}_{T_3}\neq\theta\in \Irr(T_3)$, of degree
$(q-1)(q+\sqrt{2q}+1)$. They are labeled by
\[
Z_{(k)} \text{ for } k\in E'_q,
\]
where $E'_q \subseteq\{0,1,\dots,(q-\sqrt{2q})\}$ is a set of
representatives for the equivalence relation on $\mathbb{Z}$ defined
by $c \sim c'$ if $c' \equiv c \pmod{q-\sqrt{2q}+1}$ or $c' \equiv
qc \pmod{q-\sqrt{2q}+1}$. These characters vanish on nontrivial
elements of $T_1$ and $T_2$ and a reduced value value set of
$Z_{(k)}$ is
\begin{equation}\label{eq:SuZ}
\overline{\mathcal{VS}}(Z_{(k)})=\{-(\eta^{kc}+\eta^{kcq}+\eta^{-kc}+\eta^{-kcq}):
1\le c\le q-\sqrt{2q}\},
\end{equation}
where $\eta$ is a primitive \((q-\sqrt{2q}+1)\)-th root of unity.

For the convenience of the reader, we list the possible irrational
values of the characters \(X_{(n)}\), \(Y_{(m)}\), and \(Z_{(k)}\)
in Table~\ref{table-Suzuki}.

\begin{table}[h!]
\caption{Possible irrational values of semisimple characters of $\ta
B_2(q)$.\label{table-Suzuki}}
\begin{tabular}{ccc}
\hline
Characters & Values & Parameters \\

\hline

$\begin{array}{c}
X_{(n)}\\
1\le n\le (q-2)/2
\end{array}$& $\varepsilon^{na}+\varepsilon^{-na}$ & $1\le a\le q-1$ \\

\hdashline

$\begin{array}{c} Y_{(m)}\\
m\in E_q\end{array}$& $-(\zeta^{mb}+\zeta^{mbq}+\zeta^{-mb}+\zeta^{-mbq})$ & $1\le b\le q+\sqrt{2q}$ \\

\hdashline

$\begin{array}{c} Z_{(k)} \\
k\in E'_q\end{array}$& $-(\eta^{kc}+\eta^{kcq}+\eta^{-kc}+\eta^{-kcq})$ & $1\le c\le q-\sqrt{2q}$ \\

\hline
\end{tabular}
\end{table}

%\begin{proposition}
%Theorem~\ref{theorem:main} holds for the characters $X_{(n)}$ ($1
%\le n \le (q-2)/2$) of $\ta B_2(q)$.
%\end{proposition}
%
%\begin{proof}
%    This follows from Lemma~\ref{lem:SL2q}.
%\end{proof}

\subsection{Minimal vanishing sums} Our treatment of the characters $Y_{(m)}$ and $Z_{(k)}$
will involve a case analysis where we use some basic facts about
\emph{minimal vanishing sums} of roots of unity. Given a vanishing
sum of roots of unity, we say that the vanishing sum is
\emph{minimal} if there exists no proper subsum of the sum which
also vanishes. The most basic family of examples of minimal
vanishing sums are sums over the $p$th roots of unity:
$$\sum_{i=0}^{p-1} \zeta_p^i = 0.$$

Given two sums of roots of unity, say $\sum_{i=1}^k \xi_i$ and
$\sum_{j=1}^\ell \xi_j'$, we say that the sums are \emph{equivalent up to
a rotation} if $k = \ell$ and there exists a root of unity
$\varepsilon$ such that (after a suitable reordering of the terms)
$\xi_i = \varepsilon\xi_i'$ for each $i$. Up to rotation, minimal
vanishing sums having at most twelve terms have been classified (see
\cite{PR98}). We, however, shall only require the minimal vanishing
sums having at most seven terms. These are provided in
Table~\ref{T:mvs}.

\begin{table}[h!]
\caption{Minimal vanishing sums up to a rotation having at most 7
terms. Here, $\zeta_n$ denotes a primitive $n$-th root of unity.
\label{T:mvs}}
    \begin{tabular}{c|c}
        Number of terms  & Minimal vanishing sums \\ \hline
        2  &  $1 - 1$ \\ \hdashline
        3  & $1 + \zeta_3 + \zeta_3^2$ \\ \hdashline
        4  & None \\ \hdashline
        5  & $1 + \zeta_5 + \zeta_5^2 + \zeta_5^3 + \zeta_5^4$ \\ \hdashline
        6  & $\zeta_6 + \zeta_6^5 + \zeta_5 + \zeta_5^2 + \zeta_5^3 + \zeta_5^4$ \\ \hdashline
        7  & $1 + \zeta_7 + \zeta_7^2 + \zeta_7^3 + \zeta_7^4 + \zeta_7^5 + \zeta_7^6$ \\ \hdashline
        7  & $\zeta_6 + \zeta_6^5 + \zeta_5\zeta_6 + \zeta_5\zeta_6^5 + \zeta_5^2 + \zeta_5^3 + \zeta_5^4$ \\ \hdashline
        7  & $\zeta_6 + \zeta_6^5 + \zeta_5 + \zeta_5^2\zeta_6 + \zeta_5^2\zeta_6^5 + \zeta_5^3 + \zeta_5^4$ \\ \hline

    \end{tabular}
\end{table}

\begin{lemma}
Let $q = 2^{2n+1}$ for some integer $n$, and let $\zeta$ be a
primitive $(q \pm \sqrt{2q} +1)$-th root of unity, and set
$$
    \alpha_m = \zeta^{m} + \zeta^{mq} + \zeta^{-m} + \zeta^{-mq}
$$
for an integer $m$, and assume that $\alpha_m$ cannot be expressed
as a sum of fewer then four roots of unity. Let $\sigma$ denote the
Galois automorphism of $\QQ_{c(\alpha_m)}$ sending $\zeta^{m}$ to
$\zeta^{mq}$. Then, $\QQ(\alpha_m)$ is precisely the subfield of
$\QQ_{c(\alpha_m)}$ fixed by $\sigma$.
\end{lemma}

\begin{proof}
Notice that by Lemma~\ref{L:Lox}, our assumptions imply that
$\QQ_{c(\alpha_m)} = \QQ(\zeta^{m})$, so the action of $\sigma$ on
the individual terms of $\alpha_m$ is well-defined. Also, observe that $q^2 + 1 =
(q + \sqrt{2q} + 1)(q - \sqrt{2q} + 1)$, so $\zeta^{mq^2} =
\zeta^{-m}$. In particular, the terms of $\alpha_m$ are indeed
permuted by $\sigma$.

First, assume that the terms of $\alpha_m$ are each distinct. The
first author has shown in his dissertation \cite[Theorem
3.8]{Herb26} that for any sum of four roots of unity $\beta$, the
degree $[\QQ_{c(\beta)}:\QQ(\beta)]$ is at most 6.
Thus, as
$\sigma$ permutes the terms of $\alpha_m$ transitively and has order 4,
it follows that $\Gal(\QQ_{c(\alpha_m)}/\QQ(\alpha_m))
= \langle \sigma \rangle$, and the lemma holds in this case by the
Galois correspondence.

If the terms of $\alpha_m$ are not distinct, this implies that
either $\zeta^{mq} = \zeta^{m}$ or that $\zeta^{mq} = \zeta^{-m}$.
In the former case, $\alpha_m = 4\zeta^m$, and the lemma holds
trivially in this case. In the latter case, this implies that
$\alpha_m = 2(\zeta^m + \zeta^{-m})$ and that $\sigma$ is the
automorphism of complex conjugation, but the lemma follows in this
case from Lemma~\ref{L:diss-lem}.
\end{proof}

The following key result addresses the characters \(Y_{(m)}\) and
\(Z_{(k)}\).

\begin{proposition}\label{prop:Su-vanishingsum}
Let $m$ be an integer, and let $q = 2^{2n + 1}$ for some integer $n
\geq 1$. Let $\zeta$ be a primitive $(q \pm \sqrt{2q} + 1)$-th root
of unity. Let $$S:=\{\zeta^{mb} + \zeta^{mqb} + \zeta^{-mb} +
\zeta^{-mqb} : b \in \ZZ\}.$$ Then, \[\QQ(S) = \QQ(\zeta^{m} +
\zeta^{mq} + \zeta^{-m} + \zeta^{-mq}).\]
\end{proposition}

\begin{proof}
Fix an integer $m$ and let $\zeta$ be a primitive $(q \pm \sqrt{2q}
+ 1)$-th root of unity where $q = 2^{2n + 1}$. In particular,
$\zeta$ is a root having odd order. Now, set $\alpha_m = \zeta^{m} +
\zeta^{mq} + \zeta^{-m} + \zeta^{-mq}$. If $\alpha_m$ cannot be
expressed as a sum of less than four roots of unity, then by the
result of Loxton in Lemma~\ref{L:Lox}, we have that
$\QQ_{c(\alpha_m)} = \QQ(\zeta^m, \zeta^{mq}, \zeta^{-m},
\zeta^{-mq}) = \QQ(\zeta^m)$. Let $\sigma$ denote the Galois
automorphism of $\QQ(\zeta^m)$ sending $\zeta^{m}$ to $\zeta^{mq}$.
Since all values of $S$ lie in $\QQ_{c(\alpha_m)} = \QQ(\zeta^m)$,
the proposition follows in this case from the above lemma since all
elements of $S$ are fixed by $\sigma$. For the remainder of the
proof, we shall assume that $\alpha_m$ is expressible as a sum of
less than four roots of unity.

\medskip

\emph{Case 1: $\alpha_m = 0$.} Assume that $\alpha_m = 0$ so that
$\alpha_m$ is a vanishing sum of four roots of unity. We claim that
this case is impossible. By Table~\ref{T:mvs}, there are no minimal
vanishing sums having four terms, so it is therefore possible to
partition the four terms of $\alpha$ into two minimal vanishing sums
each having two terms. This implies for instance that $-\zeta^{m}$
is equal to either $\zeta^{-m}$, $\zeta^{mq}$, or $\zeta^{-mq}$.
However, this would force at least one of the terms to have even
order, an impossibility.

\medskip

\emph{Case 2: $\alpha_m$ equals a root of unity.} Now, assume that
$\alpha_m$ is equal to a single root of unity. However, $\alpha_{m}$
is clearly real, so we obtain that $\alpha_m = \pm 1$. We therefore
obtain the following vanishing sum having five terms:
    \begin{equation} \label{E:one-rou}
        \zeta^m + \zeta^{-m} + \zeta^{mq} + \zeta^{-mq} \pm 1 = 0.
    \end{equation}
If the above vanishing sum is minimal, then by Table~\ref{T:mvs},
the four terms of $\alpha_m$ are either each primitive fifth roots
of unity or primitive tenth roots of unity. However, none of the
terms can have order 10 since this contradicts the fact that $\zeta$
has odd order. Now, it follows that for each $b \in \ZZ$,
    $$
\zeta^{mb} + \zeta^{-mb} + \zeta^{mqb} + \zeta^{-mqb} = \zeta_5^b
+\zeta_5^{2b} + \zeta_5^{3b} + \zeta_5^{4b} \in \{-1,5\}.
    $$
In particular, it follows that $\QQ(S) = \QQ$ in this
case, and the proposition holds trivially here.

On the other hand,
if the terms of \eqref{E:one-rou} can be partitioned into two
minimal vanishing sums, then one of these sums will have three terms
and the other will have two terms. In particular, one of the terms
of \eqref{E:one-rou}, say $\epsilon$, is equal to the negative of
some other term of the sum. However, since the terms of $\alpha_m$
are each odd, one of these two terms must be equal to $\pm 1$ and
the other, equal to $\mp 1$. Thus, at least one term of $\alpha_m$
is equal to $\pm 1$. However, all terms of $\alpha_m$ have equal
order, implying that $\alpha_m = \pm 4$, an impossibility.

\medskip

\emph{Case 3: $\alpha_m$ is expressible as a sum of two roots of
unity.} Now, assume that $\alpha_m$ can be rewritten as a sum of two
roots of unity, say, $-(\xi_1 + \xi_2)$. We obtain the below
vanishing sum having six terms:
    \begin{equation}\label{E:two-rou}
        \zeta^m + \zeta^{mq} + \zeta^{-m} + \zeta^{-mq} + \xi_1 + \xi_2 = 0.
    \end{equation}
We proceed in cases depending upon the ways the above sum can be
partitioned into minimal vanishing sums. As there are no minimal
vanishing sums having one or four terms, the possibilities
correspond to integer partitions of 6 not including either 1 or 4:
$6$, $3 + 3$, and $2 + 2 + 2$

If the vanishing sum is minimal, then by Table~\ref{T:mvs}, there
exists a root of unity $\epsilon$ such that $$\{\zeta^m, \zeta^{mq},
\zeta^{-m}, \zeta^{-mq},\xi_1,\xi_2\} = \{\epsilon\zeta_6,
\epsilon\zeta_6^5,\epsilon\zeta_5,\epsilon\zeta_5^2,\epsilon\zeta_5^3,\epsilon\zeta_5^4\}.$$
As previously mentioned, each of the terms of $\alpha_m$ have odd
order, so none of the terms can be equal to $\epsilon\zeta_6$ or
$\epsilon\zeta_6^5$. Thus, it follows that
$$
    \alpha_m = -(\xi_1 + \xi_2) = \epsilon\zeta_6 + \epsilon\zeta_6^5 = \epsilon,
$$
but this was
exactly the case considered in the previous paragraph.

On the other hand, \eqref{E:two-rou} may be partitioned into two minimal
vanishing sums each having three terms. Notice that if one of these
two vanishing sums is a subsum of $\alpha_m$, then $\alpha_m$ is
equal to a root of unity, a case we have already handled. Thus, we
may assume that $\alpha_m$ consists of two terms from both of the
minimal vanishing sums. We conclude that
    $$
    \{\zeta^m, \zeta^{mq}, \zeta^{-m}, \zeta^{-mq}\} = \{\xi_1\zeta_3, \xi_1\zeta_3^2, \xi_2\zeta_3, \xi\zeta_3^2\}.
    $$
Thus, the values of $S$ range through the set
    $$
\{\zeta^{mb} + \zeta^{mqb} + \zeta^{-mb} + \zeta^{-mqb} : b \in
\ZZ\} = \{\xi_1^b\zeta_3^b + \xi_1^b\zeta_3^{2b} +
\xi_2^b\zeta_3^{b} + \xi^b\zeta_3^{2b} : b \in \ZZ\}
    $$ $$
    = \{ -\xi_1^b - \xi_2^b : b \in \ZZ,\, 3 \nmid b\} \cup \{ 2\xi_1^b + 2\xi_2^b : b \in \ZZ,\, 3 \nmid
    b\}.
    $$
Now, notice that the elements of $S$ are all real, including
$-\xi_1 - \xi_2$, from which it follows that $\xi_1^{-1} = \xi_2$.
It follows that that $\QQ(S)$ is the field generated by the
set $\{\xi_1^b + \xi_1^{-b} : b \in \ZZ\}$, and this field is
$\QQ(\xi_1 + \xi_1^{-1})$ by Lemma~\ref{lem:SL2q}. The
proposition holds in this case.

If \eqref{E:two-rou} may be partitioned into three minimal vanishing
sums each having two terms, then this would force at least three of
the six terms to have even order. In particular, at least one of the
four terms of $\alpha_m$ would have even order in this case, an
impossibility.

\medskip

\emph{Case 4: $\alpha_m$ may be expressed as a sum of three roots of
unity.} Assume that $\alpha_m$ may be expressed as a sum of three
roots of unity, say $\alpha_m = -\xi_1 - \xi_2 - \xi_3$. We obtain
the minimal vanishing sum having seven terms below:
    \begin{equation}\label{E:three-rous}
        \zeta^m + \zeta^{mq} + \zeta^{-m} + \zeta^{-mq} + \xi_1 + \xi_2 + \xi_3 =
        0.
    \end{equation}
We again proceed in cases depending upon the ways the above sum can
be partitioned into minimal vanishing sums. Here the possible
partitions are $7$, $5 + 2$, and $3 + 2 + 2$.

We first consider the case where \eqref{E:three-rous} is a minimal
vanishing sum. Checking Table~\ref{T:mvs}, we see that there exists
a root of unity $\epsilon$ such that either
    \begin{equation}\label{E:mvs7:1}
\{\zeta^m, \zeta^{mq}, \zeta^{-m}, \zeta^{-mq}, \xi_1, \xi_2, \xi_3
\} = \{\epsilon\zeta_7^i : i=0,1,\ldots,6\},
    \end{equation}
    \begin{equation}\label{E:mvs7:2}
\{\zeta^m, \zeta^{mq}, \zeta^{-m}, \zeta^{-mq}, \xi_1, \xi_2, \xi_3
\} = \{\epsilon\zeta_6, \epsilon\zeta_6^5,
\epsilon\zeta_5\zeta_6,\epsilon\zeta_5\zeta_6^5,\epsilon\zeta_5^2,\epsilon\zeta_5^3,\epsilon\zeta_5^4\},
    \end{equation}
    or
    \begin{equation}\label{E:mvs7:3}
\{\zeta^m, \zeta^{mq}, \zeta^{-m}, \zeta^{-mq}, \xi_1, \xi_2, \xi_3
\} = \{\epsilon\zeta_6, \epsilon\zeta_6^5,
\epsilon\zeta_5,\epsilon\zeta_5^2\zeta_6,\epsilon\zeta_5^2\zeta_6^5,\epsilon\zeta_5^3,\epsilon\zeta_5^4\}
    \end{equation}
Since handling case \eqref{E:mvs7:3} is nearly identical to handling \eqref{E:mvs7:2}, we omit the handling of this case case.

First, we handle \eqref{E:mvs7:1} first. Notice that the rightmost set is fixed under
multiplication by a seventh root of unity, so we may assume without
loss that $\epsilon$ is a term belonging to $\alpha_m$. Now, notice
that $\epsilon^{-1}$ is also a term of $\alpha_m$, and thus
$\epsilon = \epsilon^{-1}\zeta_7^i$ for some $i$. In particular,
$\epsilon^{2} = \zeta_7^i$. In particular, it follows that
$\epsilon$ is a fourteenth root of unity. In fact $\epsilon$ is a
seventh root of unity, having odd order. We conclude that the values
of $S$ all lie in the unique real field lying between $\QQ$ and $\QQ(\zeta_7)$, and the proposition holds in
this case. In the case of \eqref{E:mvs7:2}, notice that $\alpha_m$
cannot have a term of the form $\epsilon\zeta_5^i$ for $i=2,3,4$,
since $\alpha_m$ would necessarily have a term of the form
$\epsilon\zeta_5^j\zeta_6^\ell$ where $j \in \{0,1\}$ and $\ell \in
\{1,5\}$, and this would force one of the terms of $\alpha_m$ to
have even order. However, this would imply that
    $$
\zeta^m +  \zeta^{mq} + \zeta^{-m} + \zeta^{-mq} = \epsilon\zeta_6 +
\epsilon\zeta_6^5 + \epsilon\zeta_5\zeta_6 +
\epsilon\zeta_5\zeta_6^5 = \epsilon + \epsilon\zeta_5.
    $$
In particular, $\alpha_m$ is expressible as a sum of two roots of
unity, a case we have already handled.

Now, we handle the case where \eqref{E:three-rous} is partitioned
into two minimal vanishing sums with one sum having two terms and
the other having five terms. That is, there exists $\epsilon_1$ and
$\epsilon_2$ such that
    $$
\{\zeta^m, \zeta^{mq}, \zeta^{-m}, \zeta^{-mq}, \xi_1, \xi_2, \xi_3
\} =
\{\epsilon_1,\epsilon_1\zeta_5,\epsilon_1\zeta_5^2,\epsilon_1\zeta_5^3,\epsilon_1\zeta_5^4,\epsilon_2,-\epsilon_2\}.
    $$
First, notice that one of either $\epsilon_2$ or $-\epsilon_2$ is a
term of $\alpha_m$, or otherwise, $\alpha_m$ would be expressible as
either a root of unity or a sum of two roots of unity. It is of no
loss to assume that $\epsilon_2$ is a term of $\alpha_m$. We
conclude that
    $$
\{\zeta^{m},\zeta^{mq},\zeta^{-m},\zeta^{-mq}\} =
\{\epsilon_1\zeta_5^c, \epsilon_1\zeta_5^d, \epsilon_1\zeta_5^e,
\epsilon_2 \}
    $$
for distinct $c,d,e \in \{0,1,2,3,4\}$. Now, as the left set above
is closed under conjugation, there is $c,d \in \{0,1,2,3,4\}$ such
that $(\epsilon_1\zeta_5^{c})^{-1} = \epsilon_1\zeta_5^d$. In
particular, we have that $\epsilon_1^2$ is a fifth root of unity. In
particular, since the terms of $\alpha_m$ are odd, $\epsilon_1$ is a
fifth root of unity. Also, as all terms of $\alpha_m$ have the same
order, it follows that all the terms of $\alpha_m$ are fifth roots
of unity as well, and we conclude that the values $S$ lie in the
real subfield lying between $\QQ$ and $\QQ(\zeta_5)$, and the
proposition holds in this case.

Now, we consider \eqref{E:three-rous} may be partitioned into three
minimal vanishing sums with one sum having three terms and the other
two sums both having two terms. There then exists three roots of
unity, $\epsilon_1, \epsilon_2,$ and $\epsilon_3$ such that
    $$
\{\zeta^m, \zeta^{mq}, \zeta^{-m}, \zeta^{-mq}, \xi_1, \xi_2, \xi_3
\} = \{\epsilon_1, \epsilon_1\zeta_3, \epsilon_1\zeta_3^2,
\epsilon_2, -\epsilon_2, \epsilon_3, -\epsilon_3\}.
    $$
One sees that $\alpha_m$ must have exactly one of either
$\epsilon_2$ or $-\epsilon_2$ as a term, or otherwise, $\alpha_m$ is
expressible as a sum of two or less roots of unity. We draw a
similar conclusion regarding $\epsilon_3$. In particular, we may
assume that $\epsilon_2$ and $\epsilon_3$ are terms of $\alpha_m$.
Similarly, we may assume without loss that $\epsilon_1\zeta_3$ and
$\epsilon_1\zeta_3^2$ are terms of $\alpha_m$. We have
    $$
\{\zeta^{m},\zeta^{mq},\zeta^{-m},\zeta^{-mq}\} =
\{\epsilon_1\zeta_3, \epsilon_1\zeta_3^2, \epsilon_2, \epsilon_3 \}
    $$
Now, as the left set is closed under conjugation, we see that
$(\epsilon_1\zeta_3)^{-1}$ is also a term of $\alpha_m$. If
$(\epsilon_1\zeta_3)^{-1} = \epsilon_1\zeta_3^2$, this implies that
$\epsilon_1 = \pm 1$. As all terms of $\alpha_m$ have the equal, odd
order, we conclude that the terms of $\alpha_m$ are each primitive
third roots of unity, and it follows in this case that all values of
$S$ are rational. On the other hand, it may happen that
$(\epsilon_1\zeta_3)^{-1} = \epsilon_2$ or $\epsilon_3$. It is of no
loss to assume $(\epsilon_1\zeta_3)^{-1} = \epsilon_2$. This also
implies that $(\epsilon_1\zeta_3^2)^{-1} = \epsilon_3$. One sees
that these two relations imply that $\epsilon_3 =
\epsilon_2\zeta_3^{2}$. In particular, we obtain
$$
\alpha_m = \epsilon_1\zeta_3 + \epsilon_1\zeta_3^2 + \epsilon_2 + \epsilon_2\zeta_3^j = -\epsilon_1 - \epsilon_2\zeta_3.
$$
Thus, $\alpha_m$ is a sum of two roots of unity in this case, a case which we have already handled above.

This concludes the proof.
\end{proof}

%%%%%%%%%%%%%%%%%%%%%%%%%%%%%%%%%%%
%%%%%%%%%%%%%%%%%%%%%%%%%%%%%%%%%%%

\subsection{Proof of Theorem~\ref{theorem:main} for $\ta B_2(q)$}

\begin{proposition}\label{prop:main-Suzuki}
Theorem~\ref{theorem:main} holds for all irreducible characters of
$\ta B_2(q)$.
\end{proposition}

\begin{proof}
By the discussion in Subsection~\ref{subsec:Suzuki-characters}, it
remains to consider the characters \(X_{(n)}\), \(Y_{(m)}\), and
\(Z_{(k)}\) with reduced value sets given in \eqref{eq:SuX},
\eqref{eq:SuY}, and \eqref{eq:SuZ}, respectively. The result for
\(X_{(n)}\) follows from Lemma~\ref{lem:SL2q}, while the results for
\(Y_{(m)}\) and \(Z_{(k)}\) follow from
Proposition~\ref{prop:Su-vanishingsum}.
\end{proof}

Theorem~\ref{theorem:main} is now completely proved, by
Propositions~\ref{prop:GL2q}, \ref{prop:main-SL2q}, and
\ref{prop:main-Suzuki}.

\subsection{Concluding remarks}
It is plausible that our techniques could be extended to establish
(\ref{star}) for other low-rank groups whose generic character
tables are known, such as the Ree groups $\ta G_2(q)$
(\cite{Ward62,Brunat07}), as well as the linear and unitary groups
$\GL_3(q)$, $\GU_3(q)$, $\SL_3(q)$, and $\SU_3(q)$
(\cite{Steinberg51,Ennola63,SimpsonFrame73}). However, carrying this
out would be considerably more involved and would substantially
increase the length of the paper.

It is worth noting that we were unable to find irreducible
characters $\chi$ of any group $G$ such that
\[
c(\chi) > \lcm(c(\chi(x)), c(\chi(y)) \quad \text{for all } x, y \in
G.
\]

In \cite[Question~7.3]{Navarro23}, G.~Navarro asked whether
\(\QQ(\chi)\) is generated by a single value whenever \(\chi\) has
odd degree or is \(2\)-rational; as far as we know, this remains
open.

Finally, for further discussion of the significance and applications
of the character conductor in other problems in group representation
theory, we refer the reader to the recent survey~\cite{Hung25}.

%%%%%%%%%%%%%%%%%%%%%%
%%%%%%%%%%%%%%%%%%%%%%%
%%%%%%%%%%%%%%%%%%%%%%
%%%%%%%%%%%%%%%%%%%%%%%

\end{document}